\documentclass[11pt]{article}
\usepackage{amsthm, amsmath, amssymb, amsfonts, url, booktabs, tikz, setspace, fancyhdr, bm}
\usepackage{geometry}
\usepackage{hyperref, enumerate}
\usepackage[shortlabels]{enumitem}
\usepackage[babel]{microtype}
\usepackage[english]{babel}
\usepackage[capitalise]{cleveref}
\usepackage{comment}
\usepackage{bbm}
\usepackage{csquotes}
\usepackage{mathabx}
\usepackage{tikz}
\usepackage{graphicx}
\usepackage{float}
\usepackage{xcolor}
\usepackage{mathtools}
\usetikzlibrary{positioning, arrows.meta, shapes.geometric}

\counterwithin{figure}{section}

\newtheorem{theorem}{Theorem}[section]
\newtheorem{prop}[theorem]{Proposition}
\newtheorem{conj}[theorem]{Conjecture}
\newtheorem{lemma}[theorem]{Lemma}
\newtheorem{cor}[theorem]{Corollary}

\newtheorem{fact}[theorem]{Fact}

\newtheorem{question}{Question}

\usetikzlibrary{decorations.pathmorphing}

\theoremstyle{definition}

\newtheorem*{defn-non}{Definition}

\newlist{Case}{enumerate}{2}
\setlist[Case, 1]{%
    label           =   {\bfseries Case \arabic*.},
    labelindent=1em ,labelwidth=1.3cm, labelsep*=1em, leftmargin =!
}
\setlist[Case, 2]{%
    label           =   {\bfseries Subcase \arabic{Casei}.\arabic*.},
    labelindent=-1em ,labelwidth=1.3cm, labelsep*=1em, leftmargin =!
}

\usepackage{todonotes}

\newcommand{\induced}{\mathrm{induced}}

\title{Exact Tur\'{a}n densities in triple systems } 
\author{
Nannan Chen\thanks{School of Mathematics, Shangdong University, Jinan, China, and Extremal Combinatorics and Probability Group (ECOPRO), Institute for Basic Science (IBS), Daejeon, South Korea. Supported by  China Scholarship Council and IBS-R029-C4. 
Emails: nannanchen@sdu.edu.cn, yzq\_sdu\_edu@163.com.
}
\quad\quad
Yuzhen Qi\footnotemark[1]
\and
Caihong Yang\thanks{School of Mathematics and Statistics, Fuzhou University, Fujian,
China, and Extremal Combinatorics and Probability Group (ECOPRO), Institute for Basic Science (IBS), Daejeon, South Korea.  Supported by National Key R\&D Program of China (Grant No. 2023YFA1010202), the Central Guidance on Local Science and Technology Development Fund of Fujian Province (Grant No. 2023L3003), the Institute for Basic Science (IBS-R029-C4). 
Email: chyang.fzu@gmail.com.}
\and
Hongbin Zhao\thanks{School of Mathematics and Statistics, Fuzhou University, Fujian,
China. Supported by National Key R\&D Program of China (Grant No. 2023YFA1010202), the Central Guidance on Local Science and Technology Development Fund of Fujian Province (Grant No. 2023L3003). Email: hbzhao2024@163.com.}
}
\begin{document}
\date{}
\maketitle
\begin{abstract}
In this paper, we prove several new Tur\'{a}n density results for $3$-graphs. We show:
\begin{center}
     $\pi(C_4^3, \overline{F_5}) = 2\sqrt{3}-3$, $\pi(F_{3,2}, C_5^{3-}) = \frac{2}{9}$ and $\pi(F_{3,2}, \text{induced  }\overline{F_{3,2}}) = \frac{3}{8}$. 
\end{center}
The first result confirms the conjecture of Shi [On Tur\'an denisties of small triple graphs, European J. Combin. 52 (2016) 95–102]. The other  results give several special non-principal family posed by Mubayi and R\"odl [On the Tur\'an number of triple systems, J. Combin. Theory A. 100 (2002) 135-152]. 

\medskip
    \textbf{Keywords:} Hypergraphs; Tur\'an density; Non-principal 
\end{abstract}

\section{Introduction}\label{introduction}
For a collection $\mathcal{F}$ of $r$-uniform hypergraphs ($r$-graphs), we say that an $r$-graph $H$ is \emph{$\mathcal{F}$-free} or free of $\mathcal{F}$, if $H$ contains no $\mathcal{F} \in \mathcal{F}$ as a subhypergraph. The \emph{Tur\'{a}n number} $\operatorname{ex}(n, \mathcal{F})$ is defined to be the maximum number of $r$-edges an $n$-vertex $\mathcal{F}$-free $r$-graph can have. Identifying $\operatorname{ex}(n, \mathcal{F})$ is a central problem in Extreme Combinatorics, and is also notoriously difficult when $r \geqslant 3$, where even the \emph{Turán density} $\pi(\mathcal{F}):=\lim _{n \rightarrow \infty} \operatorname{ex}(n, \mathcal{F}) /\binom{n}{r}$ is only known for a few $\mathcal{F}$'s despite significant effort devoted to this area. 

A fundamental case is when $\mathcal{F} =K_t^r$, the complete $r$-graph on $t$ vertices. 
In 1941, Tur\'an \cite{turan1941on} initiated the study of this problem  by asking: what is the maximum number of edges in an $n$-vertex $r$-graph  that does not contain a complete $r$-graph $K_t^{r}$ on $t$ vertices? 
For $r=2$, Tur\'an's theorem \cite{turan1941on} provides a complete solution. However, for $t>r\geq 3$, this problem remains open. Even the first nontrivial case of $K_4^3$ (i.e., $r=3$ and $t=4$) has yet to be resolved. Many constructions (e.g., see \cite{Brown_1983, Kostochka_1982, Fon_1988}) are known to establish the lower bound $\pi(K_4^3) \ge \frac{5}{9}$. On the other hand, the current record is $\pi(K_4^3) \le 0.5615$ \cite{baber2012turan}. For further results on Tur\'an density, we refer the reader to the survey by Balogh, Clemen and Lidick\'{y} \cite{Felix_2022} and paper by Ravry and Vaughan \cite{Victor_2013}.

 A related family that has also received extensive attention for Tur\'an density problems is the tight cycle. The $3$-uniform tight $\ell$-cycle $C_{\ell}^{3}$ is the $3$-graph on $\{1, \ldots, \ell\}$ consisting of all $\ell$ consecutive triples in the cyclic order. Denote by $C_{\ell}^{3-}$the 3-uniform hypergraph obtained by removing one hyperedge from the tight cycle on $\ell$ vertices. Balogh and Luo \cite{Balogh_Luo_2024} showed that the Tur\'an density of $C_{\ell}^{3-}$ is $1 / 4$ for every sufficiently large $\ell$ not divisible by 3. 
 Very recently, Lidick\'y-Mattes-Pfender \cite{LMP24}, and independently Bodn\'ar-Le\'on-Liu-Pikhurko \cite{XZLiu_2024} proved that $\pi(C^{3-}_\ell) = 1/4$ for every $\ell \ge 5$ satisfying $ \ell\not\equiv 0 (\text{mod }3)$.
 Let $\mathcal{C}$ be either the pair $\left\{C_{4}^{3}, C_{5}^{3}\right\}$ or the single tight $\ell$-cycle $C_{\ell}^{3}$ for some $\ell \geq 7$ not divisible by 3. Very recently, Bodn\'{a}r, Le\'{o}n, Liu and Pikhurko \cite {bodnar2025} showed that the Tur\'{a}n density of $\mathcal{C}$ is equal to $2 \sqrt{3}-3$. 

Given a hypergraph $H$, we use $\overline{H}$ to denote its complement.
Let $F_5$ be a 3-graph defined on $[5]$ with $E(F_5) = \{123, 145, 245\}$. Since $C_5^3\subseteq \overline{F_5}$, it follows from the result in \cite {bodnar2025} that 
$\pi(C_4^3, \overline{F_5}) \ge 2\sqrt{3}-3$. In \cite{SHI2016}, Shi proved an upper bound of $\pi(C_4^3, \overline{F_5})$, i.e.,  $\pi(C_4^3, \overline{F_5}) \le 2-\sqrt{2}$. 
Note that a $\{C_4^3, \overline{F_5} \}$-free triple graph is exactly a $C_4^3$-free triple graph with every five vertices spanning at most six edges. Based on this structure, Shi \cite{SHI2016} conjected that the upper bound could be improved to $2\sqrt{3}-3$. Motivated by the conjecture and the method developed in \cite{bodnar2025} , the main result of this paper confirms Shi’s prediction.  

\begin{theorem}\label{thm:main}
    $\pi(C_4^3, \overline{F_5}) = 2\sqrt{3}-3$. 
\end{theorem}

In~\cite{Mubayi_Rodl_2002}, Mubayi and R\"odl provided the following construction, which establishes the lower bound $\pi(C_\ell^3) \ge 2\sqrt{3} - 3$ for every $\ell \ge 4$ not divisible by 3. This construction also yields the same lower bound for the Tu\'an density of $\{C_4^3, \overline{F_5}\}$. 

We say that a 3-graph $H$ is \emph{semi-bipartite} if there exists a partition $V_1 \cup V_2 = V(H)$ such that every edge in $H$ contains exactly two vertices from $V_1$. Denote by $\mathbb{B}[V_1, V_2]$ the \emph{complete semi-bipartite} 3-graph with parts $V_1, V_2$, i.e.,
$$
\mathbb{B}[V_1, V_2] := \left\{ e \in \binom{V_1 \cup V_2}{3} : |e \cap V_1| = 2 \right\}.
$$

\textbf{Lower bound.} For $n \in \{0,1,2\}$, the (unique) $n$-vertex $\mathbb{B}_{\text{rec}}$-\emph{construction} is the empty 3-graph on $n$ vertices. For $n \ge 3$, an $n$-vertex 3-graph $H$ is a $\mathbb{B}_{\text{rec}}$-\emph{construction} if there exists a partition $V_1 \cup V_2 = V(H)$ into non-empty parts such that $H$ is obtained from $\mathbb{B}[V_1, V_2]$ by adding a copy of $\mathbb{B}_{\text{rec}}$-construction into $V_2$. Additionally, a 3-graph is called a $\mathbb{B}_{\text{rec}}$-\emph{subconstruction} if it is a subgraph of some $\mathbb{B}_{\text{rec}}$-construction.

Let $b_{\text{rec}}(n)$ denote the maximum number of edges in an $n$-vertex $\mathbb{B}_{\text{rec}}$-construction. By definition, for each $n \ge 3$, we have
$$b_{\text{rec}}(n) = \max \left\{ \binom{n_1}{2} n_2 + b_{\text{rec}}(n_2) : n_1 + n_2 = n \text{ and } n_1 \ge 1 \right\}.$$

Simple calculations show that, as $n \to \infty$, the optimal ratio $n_1 / n$ approaches $\frac{3 - \sqrt{3}}{2} + o(1)$, and 
$$b_{\text{rec}}(n) = (2\sqrt{3} - 3 + o(1)) \binom{n}{3}.$$ 
It is easy to verify that every $\mathbb{B}_{\text{rec}}$-construction is $\{C_4^3, \overline{F_5}\}$-free. Therefore, this construction gives a lower bound $\pi(C_4^3, \overline{F_5}) \ge 2\sqrt{3} - 3$.\\ 
\vspace{-3mm}

\textbf{Overview of the upper bound of \cref{thm:main}.} The proof employs flag algebra machinery developed by Razborov \cite{flag_2007} and is computer-assisted.

The strategy begins by establishing a structural property of dense $\{C_4^3, \overline{F_5}\}$-free 3-graphs (see Lemma~\ref{lem:key}). 
 Specifically,  every sufficiently dense $\{C_4^3, \overline{F_5}\}$-free 3-graph $H$ admits a bipartition $V_1 \cup V_2$ with $\frac{n}{2} \le |V_1| \le \frac{4n}{5}$ such that the total number of edges is at most $\binom{|V_1|}{2}|V_2| + |H[V_2]| + \xi n^3$,
up to a penalty term depending on the number of irregular edges across the partition. This structural insight plays a key role in obtaining the upper bound of Theorem~\ref{thm:main}.

\textbf{Organisation.} 
The rest of the paper is  organized as follows. 
In \cref{sec:proof of main},  we introduce a useful lemma (\cref{lem:key}) and provide the proof of \cref{thm:main}. \cref{sec:Flag} presents the results obtained via computer using Razborov’s flag algebra method, which will be used to prove \cref{lem:key}.  The proof of \cref{lem:key} is given in \cref{sec:key}. Finally, we discuss non-principal families  and pose some open questions in \cref{sec:other results}.

\section{Proof of \cref{thm:main}}\label{sec:proof of main}

To state Lemma~\ref{lem:key}, we first introduce a definition related to a 3-graph $H$. 
For a partition $V(H)=V_1 \cup V_2$, let
\begin{equation}\label{BM}
    \begin{aligned}
    B_{H}\left(V_1, V_2\right) & :=\left\{e \in H:\left|e \cap V_1\right| \in\{1,3\}\right\},  \\
    M_{H}\left(V_1, V_2\right) & :=\left\{e \in \overline{H}:\left|e \cap V_1\right|=2\right\} . 
    \end{aligned}
\end{equation}
be the sets of \emph{bad} and \emph{missing} edges respectively (see Figure \ref{fig:1}). Note that the edges inside $V_2$ are not classified as bad or missing. For convenience, we will omit $(V_1, V_2)$ and the subscript $H$ if it is clear from the context.\\ 

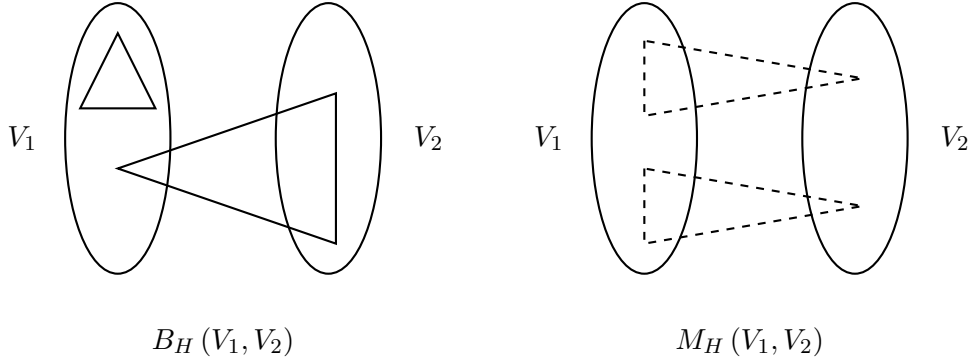
\begin{figure}[htbp]
\centering 
\begin{tikzpicture}
\draw [rotate around={90:(1.6,4.2)},line width=0.8pt] (1.6,4.2) ellipse (1.8cm and 0.7cm);
\draw [rotate around={90:(4.4,4.2)},line width=0.8pt] (4.4,4.2) ellipse (1.8cm and 0.7cm);
\draw (0,4.5) node[anchor=north west] {$V_1$};
\draw (5.4,4.5) node[anchor=north west] {$V_2$};
\draw (3,1.5) node[anchor=center] {$ B_{H}\left(V_1, V_2\right)$};

\draw[thick, fill=none]
    (1.6,5.6) -- (1.1,4.6) -- (2.1,4.6) -- cycle;
\draw[thick, fill=none]
(1.6,3.8) -- (4.5,4.8) -- (4.5,2.8) -- cycle;

\draw [rotate around={90:(8.6,4.2)},line width=0.8pt] (8.6,4.2) ellipse (1.8cm and 0.7cm);
\draw [rotate around={90:(11.4,4.2)},line width=0.8pt] (11.4,4.2) ellipse (1.8cm and 0.7cm);

\draw (7,4.5) node[anchor=north west] {$V_1$};
\draw (12.4,4.5) node[anchor=north west] {$V_2$};
\draw (10,1.5) node[anchor=center] {$ M_{H}\left(V_1, V_2\right)$};

\draw[dashed, thick, fill=none]
(8.6,5.5) -- (8.6,4.5) -- (11.5,5) -- cycle;

\draw[dashed,  thick, fill=none]
(8.6,3.8) -- (8.6,2.8) -- (11.5,3.3) -- cycle; 
\end{tikzpicture}
\caption{$ B_{H}\left(V_1, V_2\right)$ and $ M_{H}\left(V_1, V_2\right)$. The dashed edge indicates that it belongs to $\overline{H}$. } 
    \label{fig:1}
\end{figure}

The following fact is straightforward to verify. Here, 
$$
\mathbb{S}^2:=\left\{\left(x_1, x_2\right) \in \mathbb{R}^2: x_1+x_2=1 \text { and } x_i \geq 0 \text { for } i \in[2]\right\}
$$ 
is the standard \emph{$1$-dimensional simplex}.

\begin{fact}[\cite{bodnar2025}]\label{fact:dimensional-simplex}
The following inequalities hold for every $\left(x_1, x_2\right) \in \mathbb{S}^2$ with $x_2<1$:
\begin{enumerate}[(1)]
      \item $\frac{x_1^2 x_2}{2\left(1-x_2^3\right)} \leq \frac{2\sqrt{3}-3}{6}$.

    \item Suppose that $x_1 \in\left[\frac{1}{2}, 1\right]$. Then

$$
\frac{1}{2} x_1^2 x_2+\frac{2\sqrt{3}-3}{6} x_2^3 \leq \frac{2\sqrt{3}-3}{6}-\frac{1}{4}\left(x_1-\frac{3-\sqrt{3}}{12}\right)^2.
$$   
 \end{enumerate}    
\end{fact}

The following lemma will be crucial for the proof.

\begin{lemma}\label{lem:key}
    For every $\xi>0$, there exist $\delta_{\ref{lem:key}}=\delta_{\ref{lem:key}}(\xi)>0$ and $N_{\ref{lem:key}}=N_{\ref{lem:key}}(\xi)$ such that the following holds for all $n \geq N_{\ref{lem:key}}$. Suppose that $H$ is an $n$-vertex $\{C_4^3, \overline{F_5}\}$-free $3$-graph with at least $\left(\frac{\pi(C_4^3, \overline{F_5})}{6}-\delta_{\ref{lem:key}}\right) n^3$ edges. Then there exists a partition $V_1 \cup V_2=V(H)$ with $\frac{n}{2} \le |V_1| \le \frac{4n}{5}$  such that\begin{equation*}\label{equ:keylem}
        |H| \leq\binom{\left|V_1\right|}{2}\left|V_2\right|+\left|H\left[V_2\right]\right|+\xi n^3-\max \left\{\frac{|B|}{3999}, \frac{|M|}{4000}\right\},
    \end{equation*}
    where $B=B_{H}\left(V_1, V_2\right)$ and $M=M_{H}\left(V_1, V_2\right)$ are defined in (\ref{BM}).
\end{lemma}
See \cref{sec:key} for the proof of \cref{lem:key}.

Now we are ready to prove our first main result, \cref{thm:main}.

\begin{proof}[Proof of \cref{thm:main}] 
    
    Let $\beta:=\pi(C_4^3, \overline{F_5})$.  Fix an arbitrarily small $\xi>0$. Let $\delta_{\ref{lem:key}}=\delta_{\ref{lem:key}}(\xi)>0$ be the constant given by Lemma~\ref{lem:key}. By reducing $\delta_{\ref{lem:key}}$ if necessary, we may assume that $\delta_{\ref{lem:key}} \leq \xi$. Let $n$ be large enough that

\begin{equation}\label{mainthm:eq1}
\left|\operatorname{ex}(N, \{C_4^3, \overline{F_5}\})-\beta \frac{N^3}{6}\right| \leq \delta_{\ref{lem:key}} \frac{N^3}{6} \leq \xi \frac{N^3}{6}, \quad \text {for every }  N \geq \frac{n}{5}.
    \end{equation}
    
Now consider an extremal example $H$ for $\{C_4^3, \overline{F_5}\}$ on $n$ vertices. That is, $H$ is an $n$-vertex $\{C_4^3, \overline{F_5}\}$-free $3$-graph with $\operatorname{ex}(n, \{C_4^3, \overline{F_5}\})$ edges. 
Since $|H|>\left(\frac{\beta}{6}-\delta_{\ref{lem:key}}\right) n^3$, \cref{lem:key} states that there exists a partition $V_1 \cup V_2=V(H)$ with $\frac{n}{2} \le |V_1| \le \frac{4n}{5}$  such that 
\begin{equation}\label{mainthm:eq2}
        |H| \leq\binom{\left|V_1\right|}{2}\left|V_2\right|+\left|H\left[V_2\right]\right|+\xi n^3-\max \left\{\frac{|B|}{3999}, \frac{|M|}{4000}\right\},
    \end{equation}
where $B=B_{H}\left(V_1, V_2\right)$ and $M=M_{H}\left(V_1, V_2\right)$ were defined in (\ref{BM}).   Let $x_i:=\left|V_i\right| / n$ for $i \in[2]$. It follows from (\ref{mainthm:eq1}) and (\ref{mainthm:eq2}) that
    $$(\beta-\xi) \frac{n^3}{6} \leq|H| \leq\binom{\left|V_1\right|}{2}\left|V_2\right|+\left|H\left[V_2\right]\right|+\xi n^3 \leq \frac{x_1^2 x_2 n^3}{2}+(\beta+\xi) \frac{\left(x_2 n\right)^3}{6}+\xi n^3.$$
    Combining this inequality with Fact \ref{fact:dimensional-simplex} and the fact that $x_2=1-x_1 \leq 1 / 2$, we obtain
    $$\beta \leq \frac{3 x_1^2 x_2}{1-x_2^3}+\frac{\left(7+x_2^3\right) \xi}{1-x_2^3} \leq 2\sqrt{3}-3+\frac{57 \xi}{7}.$$
    Letting $\xi \rightarrow 0$, we conclude that $\pi\left(C_4^3,\overline{F_5}\right)=\beta \leq 2\sqrt{3}-3$, which proves Theorem \ref{thm:main}.
\end{proof}  

\section{Flag algebra}\label{sec:Flag}
In this section, we present the results obtained via computer using Razborov’s flag algebra method \cite{flag_2007}, as also described in \cite{BT11,SFS16,GGHLM22,Raz10}. Since this method is well-known by now, we will be very brief. In particular, we omit many definitions, referring the reader to \cite{flag_2007,Raz10} for any missing notions. Roughly speaking, a flag algebra proof using 0-flags on $m$ vertices of an upper bound $u \in \mathbb{R}$ on the given objective function $f$ consists of an identity
$$
u-f(H)=\mathrm{SOS}+\sum_{F \in \mathcal{F}_{m}^{0}} c_{F} \cdot p(F, H)+o(1)
$$
which is asymptotically true for any admissible $H$ with $|V(H)| \rightarrow \infty$. 
Here, the SOS term can be represented as a sum of squares, as described, for example, in \cite[Section 3]{Raz10}. $\mathcal{F}_{m}^{0}$ consists of all up to isomorphism 0-flags (objects of the theory without designated roots) with $m$ vertices, each coefficient $c_F \in \mathbb{R}_{\geq 0}$. $p(F, H) \in[0,1]$ is the density of $k$-subsets of $V(H)$ that span a subgraph isomorphic to $F$. If $f(H)$ can be represented as a linear combination of the densities of members of $\mathcal{F}_{m}^{0}$ in $H$, then finding the smallest possible $u$ amounts to solving a semi-definite program (SDP) with $\left|\mathcal{F}_{m}^{0}\right|$ linear constraints.

We analyzed the solutions returned by the computer, using a modified version of the SageMath package developed by Levente Bodn\'ar. This package is still under development. A short guide on how to install it and an overview of its current functionality can be found in the GitHub repository \href{https://github.com/bodnalev/sage}{https://github.com/bodnalev/sage}. The scripts that we used to generate the certificates and the certificates themselves can be found in the ancillary folder of the arXiv version of this paper or in a separate GitHub repository \href{https://github.com/HbZhao1/Flag_algebra_supplementary_files}{https://github.com/HbZhao1/Flag\_algebra\_supplementary\_files}. 

Next, we present three results (Proposition~\ref{prop:1}, \ref{prop:2} and \ref{pr:3}) that are crucial for proving \cref{lem:key}. The proofs of these results mainly rely on flag algebra techniques. Note that Proposition~\ref{prop:1}, \ref{prop:2} and \ref{pr:3} are almost identical to~\cite[Proposition 3.1, 3.2 and 3.3]{bodnar2025}.
Since the proofs are basically the same, we omit them here.

The first result states that the Tur\'an density of $\{C_4^3, \overline{F_5}\}$ is at most
$$
\alpha_{\ref{prop:1}} \coloneqq \frac{30511}{65536} \approx 0.465560913085938 \ldots  
$$
which is quite close to $2 \sqrt{3}-3 \approx 0.46410$, the value we ultimately aim to establish. 

\begin{prop}\label{prop:1}
    For every integer $n \ge 1$, every $\{C_4^3, \overline{F_5}\}$-free $n$-vertex $3$-graph $H$ has at most $\alpha_{\ref{prop:1}} \frac{n^3}{6}$ edges.
\end{prop}

For an $n$-vertex $3$-graph $H$ define the \emph{max-cut ratio} to be
\begin{align*}
    \mu(H)
    \coloneqq \frac{6}{n^3} \cdot \max\left\{|H \cap \mathbb{B}[V_1, V_2]| \colon \text{$V_1, V_2$ form a partition of $V(H)$}\right\}.
\end{align*}

The following result shows that every $\{C_4^3, \overline{F_5}\}$-free $3$-graph $H$ on $n$ vertices with large minimum degree must have a large max-cut ratio, that is, a large semi-bipartite subgraph.

\begin{prop}\label{prop:2}
    Suppose that $H$ is an $n$-vertex $\{C_4^3, \overline{F_5}\}$-free $3$-graph with $\delta(H) \ge \beta_{\ref{prop:2}}\,\frac{n^2}{2}$, where
    $\beta_{\ref{prop:2}} \coloneqq \frac{4641}{10000} \le 2\sqrt{3} - 3 - 10^{-6}.$
    Then $H$ has the max-cut ratio $\mu(H)$ at least $\alpha_{\ref{prop:2}}$, where 
    \[
    \alpha_{\ref{prop:2}}\coloneqq \frac{43099222842404160803560878960795552015959}{109451433676840840832138301242637359506224} \approx 0.393774 \ldots
    \]
\end{prop}

The key result needed for our proof is the following. We say that a partition $V(H)=V_1 \cup V_2$ of its vertex set is \emph{locally maximal} if $\left|H \cap \mathbb{B}\left[V_1, V_2\right]\right|$ does not increase when we move one vertex from one part to another. 

\begin{prop}\label{pr:3}
    For every $\xi > 0$ there exist $\delta_{\ref{pr:3}} = \delta_{\ref{pr:3}}(\xi)$ and $N_{\ref{pr:3}} = N_{\ref{pr:3}}(\xi)$ such that the following holds for all $n \ge N_{\ref{pr:3}}$.
    Let $H$ be an $n$-vertex $\{K_{4}^3,F_5\}$-free $3$-graph and $V_1 \cup V_2 = V(H)$ be a locally maximal partition.
    Suppose that $\Delta(H) - \delta(H) \le \delta_{\ref{pr:3}} n^{2}$, and $|H \cap \mathbb{B}[V_1, V_2]|\ge \alpha_{\ref{prop:2}} \frac{n^3}{6}$.
    Let  $B=B_{H}(V_1,V_2)$ and $M=M_{H}(V_1,V_2)$ be as defined in~\eqref{BM}. Then it holds that
        \begin{align}
        \label{eq:BM}
            |B|-\frac{3999}{4000}\,|M|
            \le \xi n^3.
        \end{align}
\end{prop}

\section{Proof of \cref{lem:key}}\label{sec:key}

We begin this section with two useful results that will assist in the proof of \cref{lem:key}, which is presented afterward.  
 
\begin{cor}[Proposition 2.6 in \cite{bodnar2025}]\label{pro:blowup-invariant}
 The following statements hold.
 \begin{enumerate}[(1)]
      \item For every $\delta \in\left(0,1 /6^3\right)$, there exists $n_0=n_0(\delta)$ such that every $\{C_4^3,\overline{F_5}\}$-free $3$-graph $H$ with $n \geq n_0$ vertices and at least $\left(\frac{\pi(C_4^3,\overline{F_5})}{6}-\delta\right) n^3$ edges satisfies $$\Delta(H) \leq\left(\frac{\pi(C_4^3,\overline{F_5})}{2}+10 \delta^{1 / 3}\right) n^2.$$

    \item Suppose that $H$ is an $\{C_4^3,\overline{F_5}\}$-free $3$-graph on $n$ vertices with exactly $\operatorname{ex}(n, \{C_4^3,\overline{F_5}\})$ edges. Then $\Delta(H)-\delta(H) \leq n-2$.   
 \end{enumerate}
\end{cor}

\begin{fact}[\cite{bodnar2025}]\label{fact:low-degree}
    Let $\mathcal{F}$ be a family of $r$-graphs. For every $\delta > 0$ there exists $N_0$ such that the following holds for all $n \ge N_0$. Suppose that $H$ is an $\mathcal{F}$-free $r$-graph on $n$ vertices with $|H| \ge \left( \frac{\pi(\mathcal{F})}{r!} - \delta \right) n^r$. Then the set
    \begin{equation}\label{equ:Z_delta}
        Z_\delta(H) \coloneqq \left\{ v \in V(H) : d_{H}(v) \le \left( \frac{\pi(\mathcal{F})}{(r-1)!} - 4\delta^{1/2} \right) n^{r-1} \right\}
    \end{equation}
    has size at most $\delta^{1/2} n$.
\end{fact}

Combining the two preliminary results with Propositions~\ref{prop:1}, \ref{prop:2}, and \ref{pr:3}, we are now in a position to prove \cref{lem:key}.

\begin{proof}[The proof of \cref{lem:key}.]
   Fix a sufficiently small constant $\xi>0$. Let $\delta_{\ref{pr:3}}=$ $\delta_{\ref{pr:3}}(\xi)$ and $N_{\ref{pr:3}}=N_{\ref{pr:3}}(\xi)$ be the constants given by \cref{pr:3}. 
   Take $n$  large enough so that in particular $n \geq 2 N_{\ref{pr:3}}$. 
   Choose a sufficiently small $\delta>0$  depending on $\xi, \delta_{\ref{pr:3}}$ and $N_{\ref{pr:3}}$. 
   
   Let $\beta:=\pi(C_4^3, \overline{F_5})$.  
   Suppose that $H$ is an $n$-vertex $\{C_4^3, \overline{F_5}\}$-free $3$-graph with at least $\left(\frac{\beta}{6}-\delta\right) n^3$ edges. 
   Denote $V\coloneqq V(H)$, $Z\coloneqq Z_\delta(H)$ and $U\coloneqq V(H) \backslash Z$, where $Z_\delta(H)$ was defined in (\ref{equ:Z_delta}) with $\mathcal{F}=\{C_4^3, \overline{F_5}\}$. 
   
   Consider the induced subgraph $G$ of $H$ on the vertex set $U$, where  $\hat{n}\coloneqq |U|$. 
   By \cref{fact:low-degree}, we have $\hat{n}=n-|Z| \geq\left(1-\delta^{1 / 2}\right) n \geq N_{\ref{pr:3}}$. 
   By the definition of $Z_\delta(H)$, it follows that
   $$\delta(G) \geq\left(\frac{\beta}{2}-4 \delta^{1 / 2}\right) n^2-|Z| n \geq\left(\frac{\beta}{2}-5 \delta^{1 / 2}\right) n^2 \geq\left(\frac{\beta}{2}-5 \delta^{1 / 2}\right) \hat{n}^2.$$ 
   Applying \cref{pro:blowup-invariant} to $G$ yields 
   $$\Delta(G) \leq\left(\frac{\beta}{2}+10\left(5 \delta^{1 / 2}\right)^{1 / 3}\right) \hat{n}^2.$$ 
   Therefore, 
   \begin{equation}\label{equ:5.1}
   \Delta(G)-\delta(G) \leq\left(\frac{\beta}{2}+10\left(5 \delta^{1 / 2}\right)^{1 / 3}\right) \hat{n}^2-\left(\frac{\beta}{2}-5 \delta^{1 / 2}\right) \hat{n}^2 \leq \delta_{\ref{pr:3}} \hat{n}^2.
   \end{equation}
   
   Let $V(G)=U_1 \cup U_2$ be a partition maximizing $\left|G \cap \mathbb{B}\left[U_1, U_2\right]\right|$, and set $x_i:=\left|U_i\right| / \hat{n}$ for $i \in[2]$. 
   Such a partition is maximal and hence locally maximal.
   Applying \cref{prop:2} to $G$, we obtain 
   \begin{equation}\label{equ:5.2}
   \left|G \cap \mathbb{B}\left[U_1, U_2\right]\right| = \left|G\left[U_1, U_2\right]\right| \geq \alpha_{\ref{prop:2}} \frac{\hat{n}^3}{6}, 
   \end{equation}
   noting that $\beta \geq 2\sqrt{3}-3$. 
   By combining (\ref{equ:5.2}) with the trivial upper bound $$\left|G\left[U_1, U_2\right]\right| \leq\binom{\left|U_1\right|}{2}\left|U_2\right| \leq \frac{x_1^2\left(1-x_1\right) \hat{n}^3}{2},$$ and performing exact calculations using rational numbers via computer, we deduce that 
   $\frac{52}{100} \leq x_1 \leq \frac{79}{100}$, that is, $0.52 \hat{n} \leq\left|U_1\right| \leq 0.79 \hat{n}$.
   
   Let $V(H)=$ $V_1 \cup V_2$ be an arbitrary partition satisfying $U_1 \subseteq V_1$ and $U_2 \subseteq V_2$. 
   Then we have 
   $$\begin{aligned}
   & \left|V_1\right| \geq\left|U_1\right| \geq 0.52 \hat{n} \geq 0.52\left(1-\delta^{1 / 2}\right) n \geq \frac{n}{2}, \quad \text { and } \\
   & \left|V_1\right| \leq\left|U_1\right|+|Z| \leq 0.78 n+\delta^{1 / 2} n \leq \frac{4 n}{5}.
   \end{aligned}$$ 
   Recall that $B=B_{H}\left(V_1, V_2\right)$ and $M=M_{H}\left(V_1, V_2\right)$. 
   By the definition of $B$ and $M$, we have $$|H|=\left|H\left[V_1, V_2\right]\right|+|B|+\left|H\left[V_2\right]\right|=\binom{\left|V_1\right|}{2}\left|V_2\right|-|M|+|B|+\left|H\left[V_2\right]\right|.$$ 
   Therefore, to prove \cref{equ:keylem}, it suffices to show that \begin{equation}\label{equ:5.3}
   |B|-|M| \leq \xi n^3-\max \left\{\frac{|B|}{3999}, \frac{|M|}{4000}\right\}. 
   \end{equation} 
   
   Let $\hat{B}:=B_{G}\left[U_1, U_2\right]$ and $\hat{M}:=M_{G}\left[U_1, U_2\right]$. Note that $\hat{B} \subseteq B$ and $\hat{M} \subseteq M$. Moreover,  we have 
   $$\begin{aligned}
   & |B| \leq|\hat{B}|+|Z| \cdot\binom{n}{2} \leq|\hat{B}|+\frac{\delta^{1 / 2} n^3}{2} \leq|\hat{B}|+\frac{\xi n^3}{10}, \quad \text { and } \\
   & |M| \leq|\hat{M}|+|Z| \cdot\binom{n}{2} \leq|\hat{M}|+\frac{\delta^{1 / 2} n^3}{2} \leq|\hat{M}|+\frac{\xi n^3}{10}. 
   \end{aligned}$$ 
   Hence, inequality (\ref{equ:5.3}) reduces to showing \begin{equation}\label{equ:5.4}
   |\hat{B}|-|\hat{M}| \leq \frac{\xi n^3}{2}-\max \left\{\frac{|\hat{B}|}{3999}, \frac{|\hat{M}|}{4000}\right\}. 
   \end{equation}
   
   Notice that $G$ satisfies (\ref{equ:5.1}) and (\ref{equ:5.2}), and the partition $U_1 \cup U_2=V(G)$ is locally maximal. So applying \cref{pr:3} to $G$ with parameter $\xi/4$, we obtain $|\hat{B}| \leq$ $\frac{3999}{4000}|\hat{M}|+\frac{\xi \hat{n}^3}{4}$. Consequently, 
   $$\begin{aligned} 
   & |\hat{B}|-|\hat{M}|=|\hat{B}|-\frac{3999}{4000}|\hat{M}|-\frac{|\hat{M}|}{4000} \leq \frac{\xi \hat{n}^3}{4}-\frac{|\hat{M}|}{4000}, \quad \text { and } \\
   & |\hat{B}|-|\hat{M}|=\frac{4000}{3999}\left(|\hat{B}|-\frac{3999}{4000}|\hat{M}|\right)-\frac{|\hat{B}|}{3999} \leq \frac{1000 \xi \hat{n}^3}{3999}-\frac{|\hat{B}|}{3999}.
   \end{aligned}$$
   This implies (\ref{equ:5.4}), completing the proof of \cref{lem:key}.
\end{proof}

\section{Non-principal families}\label{sec:other results}

A natural question is whether the Tur\'{a}n density of a family consisting of two 3-graphs is less than the minimum of the Tur\'{a}n densities of the individual $3$-graphs. For graphs, this is not true. Mubayi and R\"odl \cite {Mubayi_Rodl_2002} conjectured that this phenomenon does not carry over to hypergraphs and therefore posed the following conjecture. 

\begin{conj}[\cite {Mubayi_Rodl_2002}]\label{conj:non}
  There is a positive integer $k$ and $3$-graphs $F_1,\dots,F_k$ such that $\pi(\{F_1,\dots,F_k\}) < \min_i \pi(F_i)$.   
\end{conj} 

Balogh \cite{Balogh_2002} proved the conjecture, calling this phenomenon the 
\emph{non-principality} of Tur\'an function. Unfortunately, the family in Balogh \cite{Balogh_2002} has many members-significantly more than two. Mubayi and Pikhurko \cite{Mubayi_Oleg_2008}  subsequently built upon Balogh's ideas and provided the first example of a non-principal pair of 3-graphs. Later, Razborov \cite{Raz10} used the semi-definite method to show that the pair $\left(C_4^{3-}, C_5^3\right)$ is also a non-principal pair.

In this section, we exhibit a non-principal pair  for $3$-graph. Let $F_{3,2} = ([5], \{123,145,245,345\})$.



\begin{theorem}\label{thm:C5-}
   $\pi(F_{3,2}, C_5^{3-}) = \frac{2}{9} < \min \{\pi(F_{3,2}),\pi(C_5^{3-})\}.$ 
\end{theorem}
\begin{proof}{}{}
    The upper bound follows from a flag algebra calculation. 
    The lower bound is given by the balanced complete $3$-partite $3$-graph $H$. 
    Since none of $F_{3,2}$, $C_5^{3-}$ is $3$-partite, $H$ is $\{F_{3,2}, C_5^{3-}\}$-free. 

    Note that $\pi(C_5^{3-})=1/ 4$ \cite{LMP24,XZLiu_2024} and F\"uredi, Pikurkho and Simonovits \cite{Furedi_Oled_Simonovits_2005} determined $\pi\left(F_{3,2}\right)=4 / 9$. Combining these facts yields 
    $$
    \pi(F_{3,2}, C_5^{3-}) = \frac{2}{9} <\min \{\pi(F_{3,2}),\pi(C_5^{3-})\}.
    $$
\end{proof} 

It is known that for every non-complete $3$-graph $F$, the Tur\'an density of induced $F$ is trivial equal to $1$.  Consequently, for every non-complete $3$-graph $F$, we have $$\pi(F, \induced \ \overline{F}) \leq \pi(F).$$ 
In particular, if $F\subseteq \overline{F}$, then the bound is tight: $\pi(F, \induced \ \overline{F}) = \pi(F)$.  However, strict inequality can occur.  Frankl and F\"uredi \cite{Frankl_Furedi_1984} showed that $\pi(C_4^{3-},\induced \ \overline{C_4^{3-}})=\frac{5}{18}< \pi(C^{3-}_4)$,
providing a concrete example where forbidding both $F$ and its induced complement leads to a lower density. In a similar spirit, we establish a new non-principal pair $\pi(F_{3,2}, \induced \ \overline{F_{3,2}})$, which highlights another instance of such a phenomenon.

\begin{theorem}\label{thm:induced}
   $\pi(F_{3,2}, \induced \ \overline{F_{3,2}}) = \frac{3}{8} < \frac{4}{9} = \pi(F_{3,2}).$ 
\end{theorem}
\begin{proof}{}{}
    The upper bound again follows from a flag algebra calculation. 
    The lower bound is achieved by $H$ taking balanced blow-up of $K_4^3$. It is easy to check that $H$ is induced $\overline{F_{3,2}}$-free. To see that $H$ is also $F_{3,2}$-free, observe that a copy of $F_{3,2}$ cannot involve two vertices lying in the same part of $H$, while $H$ has only $4$ parts and $F_{3,2}$ has $5$ vertices. Thus, such a configuration cannot exist in  $H$. Furthermore, Mubayi and R\"odl~\cite{Mubayi_Rodl_2002} proved that $\pi(F_{3,2}) = \frac{4}{9}$, which completes the proof.  
\end{proof}

Motivated by \cref{thm:induced} and \cref{conj:non}, it is natural to raise the following question:
\begin{question}
    Is it true that for every non-complete $3$-graph $F$ with $F\not\subseteq \overline{F}$, the inequality $\pi(F, \induced \ \overline{F}) < \pi(F)$ holds? 
    If not, which classes of $3$-graphs satisfy this property?
\end{question}

\section*{Acknowledgement}
The authors would like to thank Xizhi Liu for for his technical assistance with the implementation and execution of our code.

\bibliographystyle{abbrv}
\bibliography{Turandensity}
\end{document}